\documentclass{amsart}
\usepackage{amsmath,amssymb, amsthm}
\usepackage{fullpage}
\usepackage{url}
\usepackage{cleveref}
\usepackage{array}
\usepackage{booktabs}
\usepackage{diagbox}
\usepackage{tikz}
\usepackage{graphicx} % Required for inserting images

\newtheorem{Tma}{Theorem}[section]

\newtheorem{lemma}[Tma]{Lemma}

\newtheorem{proposition}[Tma]{Proposition}

\newtheorem{theorem}[Tma]{Theorem}

\theoremstyle{definition}
\newtheorem{definition}[Tma]{Definition}
\newtheorem{example}[Tma]{Example}

\newtheorem{remark}[Tma]{Remark}

\title{Prime simplicial complexes of finite groups}
\author{Melissa Lee}
\address{School of Mathematics, Monash University, Clayton VIC 3800, Australia}
\email{melissa.lee@monash.edu}
\author{Kamilla Rekv\'enyi} 
\address{Department of Mathematics, University of Manchester, M13 9PL Manchester, UK. Also
affiliated with: Heilbronn Institute for Mathematical Research, BS8 1UG Bristol, UK.}
\email{kamilla.rekvenyi@manchester.ac.uk}
\date{\today}
\begin{document}
\begin{abstract}
The prime simplicial complex $\Pi(G)$ of a finite group $G$ is composed of all sets of primes $S$ where $G$ has an element of order the product of primes in $S$, with the subsets partially ordered by inclusion. This complex was introduced by Peter Cameron as the generalisation of the well-studied prime (or Gruenberg-Kegel) graphs.  
In this paper, we establish new results concerning two key properties of $\Pi(G)$: recognisability and purity. We demonstrate that recognisability by the prime simplicial complex is strictly stronger than recognisability by the prime graph. Notably, we present the first known example of a group that is recognisable by its prime simplicial complex and spectrum, but not by its prime graph, and is not a direct product of two isomorphic simple groups.
Furthermore, we classify groups with pure prime simplicial complexes (i.e., all maximal simplices have the same size) across several infinite families of finite simple groups. We also provide a partial classification for non-abelian finite simple groups whose prime simplicial complex has maximal simplices of size at most 2.
\end{abstract}
\maketitle

\section{Introduction}
The \textit{prime graph}, or \textit{Gruenberg-Kegel graph} $\Gamma(G)$ of a finite group $G$ has vertex set consisting of the set of primes dividing $|G|$, denoted $\pi(G)$, with edges between primes $p$, $r$ if there is an element of order $pr$ in $G$.
Prime graphs have garnered significant interest in the past few years \cite{MR4506711}. They are also amongst several different graphs on groups that have been investigated in recent years, see \cite{cameron2022graphs} for an overview.
In 2024, Cameron \cite{cameron2024simplicial} proposed generalisations of various graphs defined on groups to \textit{simplicial complexes}. In this paper, we explore the \textit{prime simplicial complex} $\Pi(G)$ of a finite group $G$, which can be considered as an intermediate between the prime graph and \textit{spectrum}, the set of all element orders, of $G$. The prime simplicial complex $\Pi(G)$ is composed of all sets of primes $S \subseteq \pi(G)$ where $G$ has an element of order the product of primes in $S$, with the subsets partially ordered by inclusion.

We investigate two distinct properties of prime simplicial complexes for finite groups. 
The first is \textit{recognisability}, which has been an area of focus for prime graphs and spectra for some time.
 A finite group $G$ is said to be \textit{$k$-recognisable} by prime graph if there are exactly $k$ groups $H$ up to isomorphism with 
$\Gamma(G) = \Gamma(H)$. We usually call a 1-recognisable group simply \textit{recognisable}, and say that $G$ is unrecognisable if it is not $k$-recognisable for any finite $k$. Cameron and Maslova \cite{MR4506711} showed that 
a group is $k$-recognisable for some finite $k$ only if it is almost simple. A significant number of results have appeared in the literature determining the recognisability for various almost simple groups, however, this problem is still open in general. The question of recognisability has also been explored for the \textit{spectrum}, or set of element orders, of finite groups. In contrast to the situation for prime graphs, there exist groups that are not almost simple but are recognisable by spectrum. The known examples are direct products of isomorphic simple groups, in particular $\mathrm{J}_4\times \mathrm{J}_4$, \cite{gorshkov2021group} and $\mathrm{J}_1\times \mathrm{J}_1$ \cite{li2024recognition}, as well as the infinite families ${}^2\mathrm{B}_2(2^m) \times {}^2\mathrm{B}_2(2^m)$ with $m=3$ or $m\geq 7$ odd \cite{wang2022criterion}, ${}^2\mathrm{G}_2(3^a) \times {}^2\mathrm{G}_2(3^a)$ with $a\geq 3$ odd \cite{li2024recognition} and $\mathrm{PSL}_n(2)^k$ with $n\geq 56k^2$ a power of two \cite{yang2023recognition}. 

 We consider the question of recognisability for prime simplicial complexes of groups in this setting. We show that the notion of recognisability here is inequivalent to that for prime graphs and spectra, and therefore provide an additional tool for recognising a group. In doing so, we also exhibit the first known example of a group which is recognisable by prime simplicial complex and spectrum that is not the direct product of isomorphic simple groups.
%Note here that if a group is recognisable by prime graph, then it is also recognisable by prime simplicial complex (see \Cref{graph_imps}).

We also explore the notion of \textit{purity} of a prime simplicial complex. The prime simplicial complex $\Pi(G)$ of a finite group $G$ is \textit{pure} if all maximal prime simplices (with respect to inclusion) to have the same cardinality. We completely determine the purity of the alternating and sporadic simple groups, as well as some results for groups of Lie type.

The paper is structured in a straightforward manner: some preliminary results are given in \Cref{sec:prelims}, the question of recognisability is considered in \Cref{sec:recog} and that of purity in \Cref{sec:purity}.

\subsubsection*{Acknowledgements} The first author  acknowledges the support of an Australian Research Council Discovery Early Career Researcher Award (project number DE230100579). The second author acknowledges the support of a Robert Bartnik Fellowship that supported her visit to Monash University, where the majority of this work was carried out. We thank Professor Peter Cameron for helpful discussions.
\section{Preliminaries}
\label{sec:prelims}
Given a finite group $G$, let $\pi(G)$ denote the primes dividing the order of $G$.
We begin by defining our main object of study.
\begin{definition}
Given a finite group $G$, the \textit{prime simplicial complex} $\Pi(G)$ of $G$ is a partially ordered set, consisting of all subsets $S = \{p_1, p_2, \dots, p_k\} \subseteq \pi(G)$ such that $G$ contains an element of order $p_1p_2 \dots p_k$. 
%The partial ordering on $\Pi(G)$ is inclusion.
\end{definition}

Each element $S$ of $\Pi(G)$ is called a \textit{prime simplex}. In a slight abuse of terminology, we will often refer to a prime simplex not only as a set of primes, but also as the product of primes contained within it.

\begin{example}\label{examples}
\begin{enumerate}
    \item Let $G=S_n$. The prime simplices of $\Pi(G)$ consist of all subsets of distinct primes $\{p_1,p_2\dots, p_k\}$ with $p_1+p_2+ \dots +p_k\leq n$.
    \item If $G$ is nilpotent, then $\Pi(G)$ consists of all subsets of $\pi(G)$.
    \item For a finite group $G$ with $r=\vert \pi(G)\vert,$ $G^r$ consist of all subsets of $\pi(G).$
\end{enumerate}

\end{example}

We record some elementary properties of prime simplicial complexes.
\begin{lemma}
Let $G$ be a group, $H\leq G$ and $N \lhd G.$ \begin{enumerate}
    \item    If $S$ is a prime simplex in $H$ then it is also a prime simplex in $G.$ 
    \item If $S$ is a prime simplex in $G/N,$ then it is also a prime simplex in $G.$
\end{enumerate}
 
\end{lemma}

\begin{lemma} \label{directprod}
    Let $G_1$ and $G_2$ be finite groups. Then $S$ is a simplex in $G=G_1 \times G_2$ if and only if there exist simplices $S_i \in G_i$ with $S=S_1 \cup S_2$. In particular, if $G_1$ is unrecognisable by prime simplicial complex, then so is $G$.
\end{lemma}
\begin{proof}
The first assertion follows directly from the fact that the order of an element $(g_1,g_2)$ in $G$ is the lowest common multiple of $|g_1|$ and $|g_2|$. Therefore, if $G_1$ and $H$ have $\Pi(G_1) =\Pi(H)$, then $\Pi(G) = \Pi(H \times G_2)$. The second assertion follows.
\end{proof}

We now highlight the connections in recognisability between the prime graph, prime simplicial complex and spectrum of a group.

\begin{lemma}
\label{graph_imps}
Let $G$ be a finite group.
    \begin{enumerate}
        \item If $G$ is recognizable by prime graph, then it is recognizable by prime simplicial complex, and by spectrum.
        \item If $G$ is unrecognizable by spectrum, then it is unrecognizable by prime simplicial complex, and by prime graph.
    \end{enumerate}
    \begin{proof}
The first result follows directly from the fact that the element orders encoded in the prime graph are contained in both the prime simplicial complex and spectrum for $G$. The second result is the contrapositive of the first.
\end{proof}
\end{lemma}
In light of \Cref{graph_imps}, we often analyse the prime simplicial complex of a group by beginning with its prime graph. We briefly recall some graph-theoretic terminology. A \textit{(co)-clique} in a graph is a subset of the vertices that are all pairwise (non)-adjacent. A vertex in a graph is \textit{universal} if it is adjacent to every other vertex in the graph.

We now proceed with some further preliminary results.

\begin{lemma}\label{centralizerproof}
    Let $G$ be finite group, $p$ be a prime divisor of $G$ and $S$ be a prime simplex of $G$ not containing $p.$ Then the following are equivalent. 
    \begin{enumerate}
        \item $\{p\} \cup S$ is a maximal simplex of $G$
        \item There exists an element $g_p$ or order $p$ such that $\{p\}\cup S$ is a maximal simplex of $C_G(g_p).$
    \end{enumerate}
\end{lemma}
\begin{proof}
First assume that $\{p\} \cup S$ is a maximal simplex of $G.$ Then there is an element $x\in G$ of order $p\prod_{s\in S} s,$ so $x^{\prod_{s\in S} s}$ has order $p$ and $x^p$ has order $\prod_{s\in S} s.$ Then $x\in C_G(x^{\prod_{s\in S} s}),$ so $\{p\} \cup S$ is a simplex of $C_G(x^{\prod_{s\in S} s})$ and it is maximal as it is maximal in $G.$ \par 
Now assume that $\{p\}\cup S$ is a maximal simplex of $C_G(g_p)$ for some element $g_p$ of order $p.$ Then $\{p\}\cup S$ is a simplex in $G.$ We show maximality in $G$ by contradiction. Assume there exists a prime $r$ such that $r\not \in\{p\}\cup S$ and $\{r\}\cup\{p\}\cup S$ is a simplex. Then repeating the same argument as earlier,  $\{r\}\cup\{p\}\cup S$ is a simplex of $C_G(g_p),$ which is a contradiction.
\end{proof}
Recall that the \textit{solvable radical} $\mathrm{Sol}(G)$ of a finite group $G$ is its unique largest normal solvable subgroup.
\begin{proposition}[{\cite[Proposition 1]{vasil2005connection}}]
\label{vas_indep_set}
 Let $G$ be a finite group and let $\mathrm{Sol}(G)$ be the solvable radical of $G$. For every subset $\rho\subseteq\pi(G)$  with $|\rho|\geq 3$ that are pairwise non-adjacent in $\Gamma(G)$, the intersection $\rho \cap \pi(\mathrm{Sol}(G))$ contains at most one prime. 
\end{proposition}

We will require the following number-theoretic results in the sections that follow.

\begin{lemma}[{\cite[Lemma A.1(i)]{BG}}]
\label{mult_order_lemma}
    Let $p\neq r$ be a prime and let $i$ be the smallest integer such that $r\mid q^i-1.$ Then $r$ divides $q^m-1$ if any only if $i$ divides $m.$
\end{lemma}

\begin{theorem}[Zsigmondy's Theorem]\label{zsigmondy}
 Let $1\leq b < a$ such that $(a,b)=1$ and $n\in \mathbb{N}.$ Then $a^n-b^n$ has a primitive prime divisor unless one of the following holds. \begin{enumerate}
     \item $a=2,$ $b=1,$ $n=1,6$,
     \item \label{zsig_square} $n=2,$ $a+b=2^m$ for some $m\geq 1.$
 \end{enumerate}   
\end{theorem}
Note that if one considers part (\ref{zsig_square}) with $b=1$, then  $a-1$ has at least two distinct prime divisors provided that $a>3$. In this case, we will abuse terminology and define $2$ to be a primitive prime divisor of $a-1$, and a primitive prime divisor of $a^2-1$ to be an odd prime divisor of $a-1$.

\begin{lemma}[{\cite{MR2076124}}]
    \label{2-3prime}
Suppose $n$ and $n+1$ are consecutive prime powers. Then one of the following holds.
\begin{enumerate}
    \item $(n,n+1) = (2^k,2^k+1)$ for some $k\geq 1$ a power of two and $n+1$ is a \textit{Fermat} prime.
    \item $(n,n+1) = (2^l-1,2^l)$ for some prime $l \geq 3$, and $n$ is a \textit{Mersenne} prime.
\item $(n,n+1)=(8,9)$.
\end{enumerate}
    % The only solutions to the equation $|2^i-3^j| = 1$ for positive integral $i,j$ are are $(i,j)=(1,1),(2,1)$ and $(3,2)$.
\end{lemma}
\begin{lemma}
[{\cite{crescenzo1975diophantine,khosravi2004alternating}}]
\label{dio_eqn}
With the exception of the solutions $(239)^2-2(13)^4=-1$ and $(3^5)-2(11)^2 = 1$, every solution of the equation
$p^a-2r^b=\pm 1$ with $p,r$ prime and $a,b>1$ has $a=b=2$.
\end{lemma}
\begin{lemma}\label{q4lemma}
    Let $q\geq 4.$ Then $\frac{q^4-1}{(4,q-1)}$ has at least 3 distinct prime divisors. 
\end{lemma}

\begin{proof}
    Note that $(q^4-1)$ factorizes as $(q^2-1)(q^2+1)=(q-1)(q+1)(q^2+1).$ By \Cref{zsigmondy}, $q-1,$ $q^2-1$ and $q^4-1$ all have a primitive prime divisor unless $q+1=2^m$ for some $m\in \mathbb{N}$ in which case $\pi (q^2-1)=\pi (q-1).$ However, if $q+1$ is a power of two, then $q-1$ has an odd prime divisor unless $q=3,$ as the only $2$ and $4$ are powers of two with difference two. 
    Note that if $q-1$ is divisible by $4$ then $q^4-1$ is divisible by $8,$ so we do not lose a prime divisor by division by $(4,q-1).$
\end{proof}

\begin{proposition}
\label{q2_divisors}
Let $q$ be a prime power such that $q^2-1$ has exactly two distinct prime divisors. Then $q \in \{4,5,7,8,9,17\}$. Moreover, if $q$ is odd, then $q^2-1$ has exactly two distinct prime divisors if and only if $(q^2-1)/4$ does.
\end{proposition}
\begin{proof}
If $q$ is even, then $q-1$ and $q+1$ are coprime, hence both must be prime powers, and in particular, one must be a power of three. Hence, $q\in \{4,8\}$ by Lemma \ref{2-3prime} as required.
Suppose now that $q$ is odd. Then $q-1$ and $q+1$ are divisible by two, hence one of them must be a power of two. If $q-1=2^k$ for some $k\geq 2$, then $k$ must be equal to three or a power of two,  $q+1 = 2(2^{k-1}+1)$, so $(2^{k-1}+1)$ must be a prime power. Hence
$k-1$ is also equal to three or a power of two by Lemma \ref{2-3prime}, so $k\in \{2,3,4\}$ and $q \in \{5,9,17\}$ as required. If instead $q+1=2^k$ for some $k \geq 2$, then $q$ is a Mersenne prime and $k$ is prime. Thus $q-1 = 2(2^{k-1}-1)$, and $(2^{k-1}-1)$ must be a prime power. So $k--1$ must also be prime and so $k=3$ and $q=7$.

Finally, if $q$ is odd, then $q^2-1$ is a multiple of 8, hence $q^2-1$ is even if and only if $(q^2-1)/4$ is.
\end{proof}
\begin{proposition}[{\cite{MR50615}}]
\label{primes}
For every integer $n$, there is a prime strictly between $n$ and $2n$. Moreover, if $n\geq 25$, there is a prime strictly between $n$ and $6n/5$.
\end{proposition}
\begin{lemma} \label{sumprimes}
Let $n\geq 3$ be an integer. The sum of the primes at most $n$ is strictly greater than $n$.
\end{lemma}
\begin{proof}

We proceed by induction. Let $P(n)$ be the statement that the sum of the primes at most $n$ is strictly greater than $n$, and let $g(n)$ denote the sum of primes at most $n$.  We verify the base cases $P(3)$, $P(4)$ directly. Now suppose $P(k)$ is true, and consider $P(k+1)$. We have $g(\lfloor \frac{k+1}{2} \rfloor) > \lfloor \frac{k+1}{2} \rfloor$. By \Cref{primes}, there is at least one prime $p$ between $\lfloor \frac{k+1}{2} \rfloor$ and $k+1$. Hence
\[
g(k+1) \geq g(\lfloor \frac{k+1}{2} \rfloor)+p > \lfloor \frac{k+1}{2} \rfloor+\lfloor \frac{k+1}{2} \rfloor+1 \geq k+1.
\]
This completes the proof.
\end{proof}

\section{Recognisability by prime simplicial complex}
\label{sec:recog}

In this section we consider the recognisability of finite groups by prime simplicial complex. Recall that the prime graph $\Gamma(G)$ of a finite group $G$ is a subset of its prime simplicial complex $\Pi(G)$. In particular, each group $H$ with $\Pi(H)=\Pi(G)$ must also have $\Gamma(H) = \Gamma(G)$. We therefore begin by recalling some results pertaining to groups with identical prime graphs.

We begin with an adapted version of \cite[Lemma 4.2]{lee2024recognisabilitysporadicgroupsisomorphism}. Recall that the \textit{Fitting subgroup} of a group $G$ is the unique largest normal nilpotent subgroup of $G$. It is generated by the largest normal nilpotent $p$-subgroups $O_p(G)$ for each prime dividing $|G|$. We will also often denote by $O_\pi(G)$ the largest nilpotent normal $\pi$-subgroup of $G$ with $\pi \subseteq \pi(G)$.
\begin{proposition}
\label{semi}
Let $G$ be a finite group with Fitting subgroup $F(G)$. Suppose $G$ contains an element of order $pr$, with $p,r$ primes. Then without loss, $p \mid |G/F(G)|$ and one of the following holds.
\begin{enumerate}
    \item $G/F(G)$ contains an element of order $pr$.
    \item  $G/F(G)$ admits an irreducible module in characteristic $r$ such that an element of order $p$ in $G/F(G)$ fixes a non-zero vector. 
\end{enumerate}
\end{proposition}

\begin{proposition}[{\cite[Theorem 7.1]{MR2213302}}]
Suppose that $G$ is a non-abelian finite simple group such that $2$ is a universal vertex in $\Gamma(G)$. Then $G$ is isomorphic to $A_n$, with $n\geq 5$ and the largest prime $p$ at most $n$ satisfying $p\leq n-3$.
\end{proposition}

\begin{proposition}[{\cite[Theorem~1]{VasilevGorshkov}}] \label{GKthmVasilev}
Suppose $G$ is a finite insoluble group such that $\Gamma(G)$ contains a vertex non-adjacent to 2. Then there exists a non-abelian simple group $S$ and soluble subgroup $K\unlhd G$ such that $S \leq G/K \leq \mathrm{Aut}(S)$.
\end{proposition}
\begin{proposition}[{\cite[Theorem 2]{MR3395066}}]
\label{solv_prime_graphs}
An unlabelled graph $\Gamma$ is isomorphic to the prime graph of a solvable group if and only if the complement $\bar{\Gamma}$ is 3-colourable and triangle-free.
\end{proposition}
\begin{proposition}[{\cite[Theorem 4.2]{MR4506711}}]
\label{num_as_grps}
Let $\pi$ be a finite set of primes.
There are at most $O(|\pi|^7)$ pairwise non-isomorphic almost simple groups $G$ such that
$\pi(G) \subseteq \pi$.
\end{proposition}
% \begin{proposition}[{\cite[Proposition 1]{vasil2005connection}}]
% Let $G$ be a finite group, and let $K$ be the solvable radical of $G$. Then for any subset $\rho \subseteq \pi(G)$ of size at least 3 with every pair of primes in $\rho$ non-adjacent in $\Gamma(G)$, we have $|\rho \cap \pi(K)|\leq 1$. In particular, if  $\Gamma(G)$ contains at least 3 pairwise non-adjacent vertices, then $G$ is insoluble. 
% \end{proposition}
We now can show the following result, which is a consequence of \cite{MR2986582}.

\begin{proposition}
\label{unrec_sol_rad}
    Let $G$ be a finite group. If $G$ is unrecognisable by prime simplicial complex, then either
    \begin{enumerate}
        \item \label{unrec_cond_1} there exists a group $H$ with non-trivial solvable radical with $\Pi(G) = \Pi(H)$, or
        \item The vertex 2 is adjacent to all other primes in $\Gamma(G)$ and $\Pi(G)$ has an odd maximal simplex.
    \end{enumerate}
Moreover, if \ref{unrec_cond_1} holds, then $G$ is unrecognisable by prime simplicial complex.
\end{proposition}

\begin{proof}
Suppose that $G$ is unrecognisable by prime simplicial complex and that for each group $H$ with $\Pi(G) = \Pi(H)$, the solvable radical of $H$ is trivial. Note that $G$ is also unrecognisable by prime graph by \Cref{graph_imps}.
If every prime simplex $s$ in $\Pi(G)$ is such that $2s \in \Pi(G)$, then $\Pi(C_2^k \times G) = \Pi(G)$ for any $k\geq 1$, which is a contradiction. Hence there exists a prime simplex $r\in \Pi(G)$ such that $2r \notin \Pi(G)$. If there exists a prime $p$ dividing $r$ such that $2p \notin \Pi(G)$, then $2p \notin \Gamma(G)$. Hence, by \Cref{GKthmVasilev} $G$ is almost simple, as is any group $H$ with the same prime graph. But by \Cref{num_as_grps}, there are only finitely many almost simple groups $H$ with $\pi(H) \subseteq \pi(G)$, contradicting the unrecognisability of $G$ by prime graph. Therefore, $2$ and $p$ must be adjacent in $\Gamma(G)$. Applying this argument as many times as necessary, we see that 2 must be adjacent to every odd prime in $\Gamma(G)$.
Finally, assume that there exists $N\unlhd G$ such that $N$ is non-trivial and solvable. Then, by \cite[Theorem]{MR2986582}, there exist infinitely many groups with the same spectrum as $G$. Each of these groups also has the same prime simplicial complex as $G$.
\end{proof}

\begin{proposition}
Suppose $G$ is $k$-recognisable by prime simplicial complex for some integer $k$. Then either $G$ is almost simple, or $2\in \Gamma(G)$ is connected to every other vertex, and $\Pi(G)$ has an odd maximal simplex.
\end{proposition}
\begin{proof}
Suppose $G$ is $k$-recognisable by prime simplicial complex for some integer $k$. Then $G$ cannot be soluble by \Cref{unrec_sol_rad}. If $\Gamma(G)$ has a prime $p$ not adjacent to 2, then \Cref{GKthmVasilev} implies that $G$ is almost simple, otherwise it is unrecognisable by \Cref{unrec_sol_rad}. The result follows.
\end{proof}
\begin{lemma}
Let $G$ be a finite group, and let $k$ be the minimum number of prime simplices of $\Pi(G)$ with union equal to $\pi(G)$. Then the direct product $G^m$ is unrecognisable by prime simplicial complex for each $m\geq k$.
\end{lemma}
\begin{proof}
If $m\geq k$, then $\Pi(G^m)$ contains a simplex equal to $\pi(G)$, and hence the simplex lattice is complete. Therefore, $\Pi(G^m)= \Pi(C)$, where $C = \cup_{p \in \pi(G)} C_p$. The result follows from \Cref{unrec_sol_rad}.
\end{proof}

\begin{proposition}
\label{dir_prod}
Let $S$ be a finite group of even order. Then the direct product $G=S^m$, $m\geq 2$ is not recognisable by prime simplicial complex.
\end{proposition}
\begin{proof}
We will show that $G$ has the same prime simplicial complex as $H=G \rtimes C$, where $C= \langle (1\,2)\rangle \leq S_m$. Let $g=(g_1, g_2, \dots, g_m, \sigma) \in H$. If $\sigma=1$, then $g \in G \leq H$, and so all prime simplices dividing the order of $g$ lie in $\Pi(G)$. So now suppose $\sigma=(1\, 2)$. For each $k \geq 1$ we have  
\[
g^{2k} = ((g_1g_2)^k, (g_2g_1)^k, g_3^{2k}, \dots g_m^{2k}, 1)
\]
and 
\[
g^{2k+1} = ((g_1g_2)^kg_1, (g_2g_1)^kg_2, g_3^{2k+1}, \dots g_m^{2k+1}, \sigma)
\]
Noting also that $g_1g_2$ and $g_2g_1$ are conjugate and thus of the same order, we deduce that the order of $g$ divides $2 \times \mathrm{LCM}(|g_1g_2|,|g_3|,\dots |g_m|)$. Now, letting $i \in G$ be an involution, we find that  $g' = (i, g_1g_2, g_3, \dots g_m,1) \in G$ has order a multiple of $|g|$. Hence, $\Pi(H) = \Pi(G)$ as required.
\end{proof}
\begin{remark}
None of the known groups that are not almost simple but recognisable by spectrum are recognisable by prime simplicial complex by \Cref{dir_prod}. This demonstrates that the converse to \Cref{graph_imps} does not hold.
It is unknown if the groups in \Cref{dir_prod} are \textit{unrecognisable} by prime simplicial complex, or just not recognisable. 
\end{remark}
\begin{proposition}
\label{recog_psc}
The group $G = \mathrm{PSL}_2(173)\times \mathrm{PSL}_2(283)$ is recognisable by prime simplicial complex.
\end{proposition}
\begin{proof}
We begin by exhibiting the prime graphs $\Gamma(\mathrm{PSL}_2(173))$ and $\Gamma(\mathrm{PSL}_2(283))$:
\begin{figure}[h!]
    \centering

\begin{tikzpicture}[scale=1.2, every node/.style={font=\small}]
  % Vertices as small solid dots
  \node[fill=black, circle, inner sep=2pt, label=left:2] (2) at (0,0) {};
  \node[fill=black, circle, inner sep=2pt, label=left:3] (3) at (1.5,0) {};
  \node[fill=black, circle, inner sep=2pt, label=left:29] (29) at (1.5,1.5) {};
  \node[fill=black, circle, inner sep=2pt, label=left:43] (43) at (0,1.5) {};
  \node[fill=black, circle, inner sep=2pt, label=above:173] (173) at (3,0.75) {};

  % Edges
  \draw (2) -- (43);
  \draw (3) -- (29);
\end{tikzpicture}
\hspace{3cm}
\begin{tikzpicture}[scale=1.2, every node/.style={font=\small}]
  % Vertices as small solid dots
  \node[fill=black, circle, inner sep=2pt, label=left:2] (2) at (0,0) {};
  \node[fill=black, circle, inner sep=2pt, label=left:3] (3) at (1.5,0) {};
  \node[fill=black, circle, inner sep=2pt, label=left:47] (29) at (1.5,1.5) {};
  \node[fill=black, circle, inner sep=2pt, label=left:71] (43) at (0,1.5) {};
  \node[fill=black, circle, inner sep=2pt, label=above:283] (173) at (3,0.75) {};

  % Edges
  \draw (2) -- (43);
  \draw (3) -- (29);
\end{tikzpicture}
\caption{The prime graphs of $\mathrm{PSL}_2(173)$ (left) and $\mathrm{PSL}_2(283)$ (right).}
\end{figure}
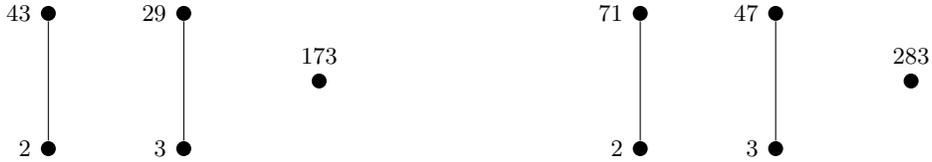

The simplices of $\mathrm{PSL}_2(173)$ and $\mathrm{PSL}_2(283)$ have size at most two, so their non-trivial simplices are exactly the vertices and edges of their prime graphs.
By \Cref{directprod} the simplices of $G$ are then of the form $S = S_1 \cup S_2$, with $S_1 \in \Pi(\mathrm{PSL}_2(173))$ and $S_2 \in \Pi(\mathrm{PSL}_2(283))$. 

Let $H$ be a group with $H \not\cong G$ but $\Pi(H) = \Pi(G)$. We begin by determining possible candidates $K$ for a non-abelian composition factor of $H$. Clearly any such $K$ must have $\pi(K) \subseteq \pi(H)$, so in particular, $5 \nmid |H|$. The non-abelian finite simple groups with order coprime to 5 are recorded in \cite[Table 1]{zavarnitsine2006recognition}, and reproduced here in \Cref{tab:co5}. In particular, $K$ must be a group of Lie type over $\mathbb{F}_q$, with $q=p^e$ and $p \in \pi(G)$.

In each case, $q-1$ divides $|K|$ and hence any (primitive) prime divisors of $q-1$ must belong to $\pi(G)$. This immediately implies $p\notin \{29,43,47,71\}$.
If $r \in \pi(G)$ divides $q^j-1=p^{je}-1$, then by \Cref{mult_order_lemma}, $je$ must be a multiple of the multiplicative order $i(p,r)$ of $p$ modulo $r$, and in particular have $p^{i(p,r)}-1$ as a divisor. We record each remaining possibility for $p,r$ along with $i(p,r)$ in \Cref{tab:mult_order}. 
In most cases, $\pi(p^{i(p,r)}-1) \not\subseteq \pi(G)$, so cannot arise. 

We now consider each option for $K$ in turn.
Firstly, $K \neq {}^3D_4(q), \mathrm{G}_2(q)$, since each has $q^6-1$ as a divisor, and hence has order divisible by $7\notin \pi(G)$.
If $K = {}^2\mathrm{G}_2(q)$ with $q=3^{2k+1}$ and $k>0$, then, by \Cref{tab:mult_order}, $q-1$ must only be divisible by 2, which is a contradiction by Lemma \ref{2-3prime} since $q>3$.
Suppose now that $K = \mathrm{PSL}_3(q)$ or $\mathrm{PSU}_3(q)$, and set $\epsilon=1$ and $-1$ respectively. Then $|K|$ is divisible both by a primitive prime divisor of $q-1$, and a primitive prime divisor of $q^{3(1-\epsilon)/2}-1$. Hence $p\neq 2,3$. Noting that $(p^3-\epsilon)\mid (q^3-\epsilon)$ allows us to eliminate $p= 173,283$ directly.

Finally, suppose $K = L_2(q)$. Noting that $(p^2-1) \mid (q^2-1)$ and that $|K|$ is divisible by at least three primes, we see that we must have either $K \cong \mathrm{PSL}_2(173)$, or $K\cong \mathrm{PSL}_2(283)$.

So each direct factor of the socle of $H$ is abelian, or isomorphic to one of $\mathrm{PSL}_2(173)$ or $\mathrm{PSL}_2(283)$.

Consider $H/\mathrm{Sol}(H)$, the quotient of $H$ by its solvable radical.
Note that $\{29,43,173\}$ and $\{47,71,283\}$ form co-cliques in $\Gamma(H) \cong \Gamma(G)$, so at most one prime from each set divides $|\mathrm{Sol}(H)|$ by \Cref{vas_indep_set}.
 By the Correspondence Theorem, $H/\mathrm{Sol}(H)$ has no solvable subgroups, so every factor of the socle is non-abelian. In particular, 
$\mathrm{PSL}_2(173)\times \mathrm{PSL}_2(283)\leq \mathrm{soc}(H/\mathrm{Sol}(H))$, otherwise $\mathrm{Sol}(H)$ violates \Cref{vas_indep_set}.

Indeed, $\mathrm{PSL}_2(173)\times \mathrm{PSL}_2(283)=  \mathrm{soc}(H/\mathrm{Sol}(H))$, otherwise any further direct factor extends one of the simplices $\{173,283\}$ or $\{2,3,29\}$ into a larger simplex that does not lie in $\Pi(G)$. 
Now $\mathrm{soc}(H/\mathrm{Sol}(H))\cong \mathrm{PSL}_2(173)\times \mathrm{PSL}_2(283)$ acts on the chief factors of $\mathrm{Sol}(H)$, and we can consider this as a vector space action. Computing the modular character tables for $\mathrm{PSL}_2(173)\times \mathrm{PSL}_2(283)$ directly in Magma, we see that elements of orders $2059 =29\times 71$ and $2021=43\times 47$ fix a non-trivial subspace in every characteristic $r \in \pi(G)$. Hence, if $\mathrm{Sol}(H)$ is non-trivial, then $H$ has a simplex containing either $\{29,71\}$ or $\{43,47\}$ that does not lie in $\Pi(G)$. Hence $\mathrm{Sol}(H)$ is trivial.
Finally, we check directly that no proper normaliser of $\mathrm{PSL}_2(173)\times \mathrm{PSL}_2(283)$ has the same prime simplicial complex. Hence $H\cong G$ and $G$ is recognisable by prime simplicial complex. 

\end{proof}

\begin{remark}
\begin{enumerate}
    \item The group $G = \mathrm{PSL}_2(173)\times \mathrm{PSL}_2(283)$ was chosen in \Cref{recog_psc} for the relative lack of small prime divisors of its order. This was done to severely limit the possibilities for composition factors for a group with the same prime simplicial complex. We anticipate this argument may be repeated for other direct products $\mathrm{PSL}_2(q_1)\times \mathrm{PSL}_2(q_2)$ with carefully chosen $q_1$, $q_2$.
\item Combining \Cref{graph_imps} and \Cref{recog_psc}, we find that $\mathrm{PSL}_2(173)\times \mathrm{PSL}_2(283)$ is recognisable by spectrum. As far as we are aware, this is the first example of a recognisable-by-spectrum group that is not a direct product of isomorphic simple groups.
\end{enumerate}

\end{remark}
\begin{table}[h]

\centering
\caption{Non-abelian simple groups of order coprime to 5. \label{tab:co5}}
\begin{tabular}{@{}lll@{}}
\toprule
$G$ & restrictions & $|G|$ \\
\midrule
$\mathrm{PSL}_2(q)$ & $3 < q \equiv \pm2 \pmod{5}$ & $\dfrac{1}{(2, q - 1)}\, q(q^2 - 1)$ \\
$\mathrm{PSL}_3(q),\, \varepsilon \in \{ +, - \}$ & $2 < q \equiv \pm2 \pmod{5}$ & $\dfrac{1}{(3, q - \varepsilon 1)}\, q^3 (q^3 - \varepsilon 1)(q^2 - 1)$ \\
$\mathrm{PSU}_3(q),\, \varepsilon \in \{ +, - \}$ & $2 < q \equiv \pm2 \pmod{5}$ & $\dfrac{1}{(3, q - \varepsilon 1)}\, q^3 (q^3 - \varepsilon 1)(q^2 - 1)$ \\
$\mathrm{G}_2(q)$ & $2 < q \equiv \pm2 \pmod{5}$ & $q^6 (q^6 - 1)(q^2 - 1)$ \\
${^2\mathrm{G}_2}(q)$ & $3 < q = 3^{2k+1}$ & $q^3 (q^3 + 1)(q - 1)$ \\
${^3D_4}(q)$ & $q \equiv \pm2 \pmod{5}$ & $q^{12}(q^8 + q^4 + 1)(q^6 - 1)(q^2 - 1)$ \\
\bottomrule
\end{tabular}
\end{table}

% By \Cref{solv_prime_graphs}, $H$ cannot be solvable, as $\{29,43,173\}$ forms a co-clique in $\Gamma(H)$. 

%\begin{table}[h!]
%\begin{tabular}{c|cccccccc}
%\backslashbox{$p$}{$r$} & 2 & 3 & 29 & 43 & 47 & 71 & 173 & 283 \\
%\hline
%2   & -- & 2  &\cellcolor{red!10} 28 & \cellcolor{red!10}14 & \cellcolor{red!10}23 & \cellcolor{red!10}35 & \cellcolor{red!10}172 & \cellcolor{red!10}94 \\
%3   & 1 & --  & \cellcolor{red!10}28 & \cellcolor{red!10}42 & \cellcolor{red!10}23 & \cellcolor{red!10}35 & \cellcolor{red!10}172 & \cellcolor{red!10}282 \\
%173 & 1 & 2  & 2  & 1  & \cellcolor{red!10}23 & \cellcolor{red!10}70 & -- & \cellcolor{red!10}282 \\
%283 & 1 & 1  & \cellcolor{red!10}14 & \cellcolor{red!10}21 & 1  & 2  & \cellcolor{red!10}172 & -- \\
%\end{tabular}
%\caption{Table of multiplicative orders of primes $p$ modulo $r$. If the cell is shaded, then $p^i-1$ has a divisor that does not lie in $\{2,3,29,43,47,71,173,283\}$.}
%\label{tab:mult_order}
%\end{table}

\begin{table}[h!]
\begin{tabular}{c|cccccccc}
\backslashbox{$p$}{$r$} & 2 & 3 & 29 & 43 & 47 & 71 & 173 & 283 \\
\hline
2   & -- & 2  & \textbf{28} & \textbf{14} & \textbf{23} & \textbf{35} & \textbf{172} & \textbf{94} \\
3   & 1  & -- & \textbf{28} & \textbf{42} & \textbf{23} & \textbf{35} & \textbf{172} & \textbf{282} \\
173 & 1 & 2  & 2  & 1  & \textbf{23} & \textbf{70} & -- & \textbf{282} \\
283 & 1 & 1  & \textbf{14} & \textbf{21} & 1  & 2  & \textbf{172} & -- \\
\end{tabular}
\caption{Table of multiplicative orders of primes $p$ modulo $r$. The bolded entries indicate that $p^i-1$ has a divisor that does not lie in $\{2,3,29,43,47,71,173,283\}$.}
\label{tab:mult_order}
\end{table}

% \begin{proposition}
% Let $G$ be a solvable group. Then $G$ is unrecognisable by prime simplicial complex.
% \end{proposition}
% \begin{proof}
% Follows directly from \Cref{unrec_sol_rad}.
% \end{proof}
% \begin{proposition}
% \label{suz_unrecog}
% The Suzuki groups $Sz(2^{2m+1})={}^2\mathrm{B}_2(2^{2m+1})$, $m\geq 1$ are unrecognisable by prime simplicial complex.
% \end{proposition}
% \begin{proof}
% By \cite[Lemma 3.6]{MR1983368}, there exists a 4-dimensional irreducible module $V$ for $S={}^2\mathrm{B}_2(2^{2m+1})$ where all elements of odd order act fixed-point-freely. Hence $V:S$ has the same simplicial complex as $S$, and therefore $S$ is unrecognisable by \Cref{unrec_sol_rad}.
% \end{proof}

% \begin{remark}
% Note that \Cref{directprod} and \Cref{suz_unrecog} imply that $Sz(2^{2m+1})\times Sz(2^{2m+1})$ are unrecognisable by prime simplicial complex. However, for $m=1$ and $m\geq 3$, these groups are recognisable by spectrum \cite[Theorem 3]{MR4586307}.
% \end{remark}
The remainder of this section is dedicated to determining the recognisability of the sporadic simple groups by prime simplicial complex.
Let 
\[
\mathcal{A} = \{\mathrm{HN},\mathrm{HS},\mathrm{M}_{11}, \mathrm{Fi}_{22}, \mathrm{Co}_3, \mathrm{He}, \mathrm{J}_2, \mathrm{M}_{12},\mathrm{McL} \}.
\]
By \cite[Table 1]{MR4529896}, a sporadic simple group is recognisable by prime graph if and only if it does not appear in $\mathcal{A}$. 
Our main result here is as follows.
\begin{theorem}
\label{spor_recog}
Let $S$ be a sporadic simple group. Then:
\begin{enumerate}
    \item $S$ is unrecognisable by prime simplicial complex if and only if $S \in \{\mathrm{M}_{12},\mathrm{J}_2, \mathrm{Co}_3, \mathrm{McL}\}$,
    \item $S$ is 2-recognisable by prime simplicial complex if and only if $G \in \{\mathrm{HS},\mathrm{M}_{11}, \mathrm{Fi}_{22}\}$, and
    \item $S$ is recognisable otherwise.
\end{enumerate} 

% either:
% \begin{enumerate}
%     \item $G \notin \mathcal{A}$ and $k=1$, 
%     \item $G \in \{\mathrm{He,\mathrm{HN}}\}$ and $k=1$, or
%     \item $G \in \{\mathrm{HS},\mathrm{M}_{11}, \mathrm{Fi}_{22}\}$ and $k=2$.
% \end{enumerate}
\end{theorem}
\begin{remark}
\Cref{spor_recog} demonstrates that the prime simplicial complexes of groups can distinguish groups that have identical prime graphs. Indeed, $\mathrm{He}$ is unrecognisable by prime graph by \cite{MR4057271}, but is recognisable by prime simplicial complex.
\end{remark}
We now proceed by considering each of the sporadic simple groups lying in $\mathcal{A}$. We will make repeated use of Proposition \ref{semi}
% \begin{table}[h!]
%     \centering
%     \begin{tabular}{cc}
%     \toprule
%       Group   & Recognisability   \\ \midrule
%          $\mathrm{M}_{11}$   &  2-recognisable with $\mathrm{PSL}_2(11)$\\
%          $HS$ &  2-recognisable with $\mathrm{PSU}_6(2)$\\
%          $HN$ & Recognisable -- $HN.2$ has an element of order $42=2*3*7$.\\
%       $Fi_{22}$ & 2-Recognisable, with $Suz.2$.\\
%       $Fi_{22}.2$& Recognisable, has an element of order $42=2*3*7$ but $Fi_{22}$, $Suz.2$ don't. \\
%       $\mathrm{M}_{12}$& unrecognisable with $11^{2n}:SL(2,5),$ $11^{2n}:SL(2,5).2$ by Hagie  \\
%       $\mathrm{J}_2$& unrecognisable with $(W^k):Sp_6(2)$ where $W$ is the natural module of $Sp_6(2)$\\
%       \bottomrule
%     \end{tabular}
%     \caption{Recognisability of some groups by prime simplicial complex}
%     \label{tab:my_label1}
% \end{table}
\begin{lemma}
\Cref{spor_recog} holds for $G \in \{\mathrm{HN},\mathrm{HS},\mathrm{M}_{11}, \mathrm{Fi}_{22}\}$.
\end{lemma}
\begin{proof}
 Each of these groups is $k$-recognisable by prime graph for some $k\in \{2,3\}$, with the precise details recorded in \Cref{krecog_spor}. Since, for a group $H$, having $\Gamma(H) = \Gamma(G)$ is necessary but not sufficient for having $\Pi(H) = \Pi(G)$, we check the candidates in \Cref{krecog_spor} directly. In particular, we find that $\mathrm{HN}$ is recognisable by prime simplicial complex since $\mathrm{HN}.2$ has an element of order 42, while $\mathrm{HN}$ does not, and $\mathrm{HS}$ remains 2-recognisable, as does $\mathrm{M}_{11}$. Finally, $\Pi(\mathrm{Fi}_{22}) \neq \Pi(\mathrm{Fi}_{22}).2$ since the latter has an element of order 42, but $\Pi(\mathrm{Fi}_{22}) = \Pi(\mathrm{Suz}.{2})$, making $\mathrm{Fi}_{22}$ 2-recognisable.
\end{proof}

\begin{table}[h!]
    \centering
    \begin{tabular}{cc}
    \toprule
       $G$  & $H$ with $\Gamma(H) = \Gamma(G)$ \\
       \midrule
        $\mathrm{HN}$ & $\mathrm{HN}.2$\\
        $\mathrm{HS}$ & $\mathrm{PSU}_6(2)$\\
        $\mathrm{M}_{11}$ & $\mathrm{PSL}_2(11)$\\
        $\mathrm{Fi}_{22}$& $H \in \{\mathrm{Fi}_{22}.2, \mathrm{Suz}.2\} $\\
        \bottomrule
    \end{tabular}
    \caption{The $2$- and $3$-recognisable sporadic simple groups.}
    \label{krecog_spor}
\end{table}
\begin{lemma}
The Held group $\mathrm{He}$ is unrecognisable by prime graph, but recognisable by prime simplicial complex.
\end{lemma}
\begin{proof}
 The group is unrecognisable by prime graph by \cite[Theorem 4]{MR4057271}. By \cite[Theorem 3(4)]{MR1995543}, if $G$ is a finite group with $\Gamma(G) \cong \Gamma(\mathrm{He})$, then $G/F(G)$ is isomorphic to one of: $\mathrm{L}_2(16)$, $\mathrm{L}_2(16).2$, $\mathrm{L}_2(16).4$, $O_8^-(2)$, $O_8^-(2).2$,  $\mathrm{\rm{PSp}}_8(2)$, $\mathrm{He}.2$ or $\mathrm{He}$. Kondrat'ev \cite[p. 84]{MR4057271} shows that the first three cases do not give rise to groups with the same prime graph as $\mathrm{He}$. The groups $O_8^-(2)$, $O_8^-(2).2$,  $\mathrm{\rm{PSp}}_8(2)$ contain an element of order 30, which is not present in $\mathrm{He}$, so these do not arise. Moreover, $\mathrm{He}.2$ has an element of order 42, so that also does not arise. Hence $\Gamma(G)/F(G) \cong \mathrm{He}$. Inspecting the $p$-modular character tables of $\mathrm{He}$, we observe that if $p \neq 17$, then elements of order 17 do not act fixed-point-freely, and if $p=17$, then elements of order 2 do not act fixed point freely. So we must have $F(G)$ being trivial, and $G \cong \mathrm{He}$.
\end{proof}
\begin{lemma}
 The group $\mathrm{M}_{12}$ is unrecognisable by prime simplicial complex.
\end{lemma}
\begin{proof}
The simplices of $\Pi(\mathrm{M}_{12})$ are $\{2,3,5,6,10,11\}$.
By \cite[Theorem 3(5)]{MR1995543}, $\Gamma(\mathrm{M}_{12}) = \Gamma(2^{10}:\mathrm{M}_{11})$, so the result follows by \Cref{unrec_sol_rad}.
\end{proof}
\begin{lemma}
The group $\mathrm{Co}_3$ is unrecognisable by prime simplicial complex.
\end{lemma}
\begin{proof}
The prime simplicial complex of $\mathrm{Co}_3$ contains $2, 3, 5, 6, 7, 10, 11, 14, 15, 21, 22, 23, 30 $. In particular, for every odd number $r$ that appears as a simplex, $2r$ is also a simplex except if $r=21,23$. Inspecting the 2-modular character table of $\mathrm{Co}_3$, we see that there are two 11-dimensional irreducible modules $V,V'$ for $\mathrm{Co}_3$ where elements of orders 21, 23 act fixed-point-freely. Hence $(V^k):\mathrm{Co}_3$ has the same prime simplicial complex as $\mathrm{Co}_3$ for each $k \geq 1$. Therefore, $\mathrm{Co}_3$ is unrecognisable.
\end{proof}

\begin{lemma}
The group $\mathrm{McL}$ is unrecognisable by prime simplicial complex.
\end{lemma}
\begin{proof}
The prime simplicial complex of $\mathrm{McL}$ is composed of simplices $2, 3, 5, 6, 7, 10, 11, 14, 15, 30$. In particular, for every odd number $r$ that appears as a simplex, $2r$ is also a simple except if $r=11$.
By direct calculation, the group $HS.2$ has the same prime simplicial complex as $\mathrm{McL}$. Inspection of the 2-modular character table of $\mathrm{HS}.2$ reveals that there is an irreducible 20-dimensional module $V$ of $\mathrm{HS}.2$ where elements of order 11 act fixed-point-freely. Hence $(V^k):\mathrm{HS}.2$ has the same prime simplicial complex as $\mathrm{McL}$ for each $k\geq 1$. The result follows.
\end{proof}

\begin{lemma}
The group $\mathrm{J}_2$ is unrecognisable by prime simplicial complex.
\end{lemma}
\begin{proof}
The prime simplicial complex of $\mathrm{J}_2$ is composed of simplices $2,3,5,6,7,10,15$. Now, let $H = \mathrm{Sp}_6(2)$. Its prime simplicial complex is identical to that of $\mathrm{J}_2$. Moreover, the natural module $W$ for $H$ is irreducible, and elements of orders 7 and 15 do not fix any non-trivial subspaces of $W$. Hence, $(W^k):H$ has the same prime simplicial complex as $\mathrm{J}_2$ for each $k\geq 1$. 
\end{proof}

This completes the proof of \Cref{spor_recog}.

\section{Purity of the prime simplicial complex}
\label{sec:purity}
Recall from the introduction that the prime simplicial complex $\Pi(G)$ of a finite group $G$ is said to be \textit{pure} if all maximal simplices have the same cardinality.

We begin with providing some examples.
\begin{example}
\label{ex_pure}
\begin{enumerate}
    \item Recall from \Cref{examples} that if $G$ is nilpotent, then $\Pi(G)$ is pure, since there is a unique maximal simplex, namely $\pi(G)$. 
    \item If $r=\vert \pi(G)\vert$ is finite, then $\Pi(G^r)$ has a unique maximal simplex $\pi(G)$, hence is pure.
    \item\label{ex_eppo} The groups with pure prime simplicial complex and maximal simplex of size 1 are exactly those where every non-trivial element of prime power order. These groups are called \textit{EPPO groups} and were classified in \cite[Theorem 1.7]{MR4506711}.
\end{enumerate}

\end{example}

We now consider the purity of the prime simplicial complexes of sporadic simple groups, which may be readily established using the Character Table Library in GAP \cite{GAP4, GAPchartab}.
\begin{proposition}\label{sporpurity}
    Let $S$ be a sporadic simple group. Then $\Pi(S)$ is not pure. The largest size of a maximal simplex and the smallest size of a maximal simplex for these groups is in Table \ref{largest_simplex}.
\end{proposition}

\begin{table}[h!]
    \centering

\renewcommand{\arraystretch}{2} % Increases row height for better vertical centering

% \begin{tabular}{p{5cm}|>{\centering\arraybackslash}m{2cm}|>{\centering\arraybackslash}m{2cm}}  
%     Group   & Largest size of simplex & Smallest size of maximal simplex  \\ \midrule
%     $\mathrm{M}_{12}, \mathrm{J}_2, \mathrm{M}_{24}, \mathrm{M}_{11}, \mathrm{M}_{22}, \mathrm{M}_{23},$  
%     $\mathrm{J}_1, \mathrm{J}_3, \mathrm{J}_4, HS, He, Ru, Suz, O'N$ & 2 & 1 \\ \midrule
%     $HN, Fi_{22}$ & 3 & 2 \\ \midrule
%     $M, B, Th, Ly, \mathrm{McL}, Fi_{24}', Fi_{23},$  
%     $\mathrm{Co}_2, \mathrm{Co}_1, \mathrm{J}_4, \mathrm{Co}_3$ & 3 & 1 \\ \midrule
% \end{tabular}
\begin{tabular}{p{5cm}cc}  
    \toprule
    Group & Largest size of simplex & Smallest size of maximal simplex \\ 
    \midrule
    $\mathrm{M}_{12}, \mathrm{J}_2, \mathrm{M}_{24}, \mathrm{M}_{11}, \mathrm{M}_{22}, \mathrm{M}_{23},$  
    $\mathrm{J}_1, \mathrm{J}_3, \mathrm{J}_4, \mathrm{HS}, \mathrm{He}, \mathrm{Ru}, \mathrm{Suz}, \mathrm{O'N}$ & 2 & 1 \\  
    $\mathrm{HN}, \mathrm{Fi}_{22}$ & 3 & 2 \\  
    $\mathrm{M}, \mathrm{B}, \mathrm{Th}, \mathrm{Ly}, \mathrm{McL}, \mathrm{Fi}_{24}', \mathrm{Fi}_{23},$  
    $\mathrm{Co}_2, \mathrm{Co}_1, \mathrm{J}_4, \mathrm{Co}_3$ & 3 & 1 \\  
    \bottomrule
\end{tabular}
\caption{Maximum and minimum sizes of a prime simplex for a sporadic group.}
    \label{largest_simplex}
\end{table}
We now consider the alternating and symmetric groups.
\begin{theorem} \label{snpurity}
    The prime simplicial complex of $S_n$ is pure if and only if $n\leq 4$ or $n=9.$ 
\end{theorem}

\begin{proof} 
%For $n\leq 4$ we can see that the prime simplicial complex is pure with largest simplex size $1$. For $n=9,$

%For $n\leq 129,$ we check the statements using GAP.
%Now let $n\geq 129.$
Assume that the prime simplicial complex of $S_n$ is pure of length $k$. Let $p\leq n$ be the largest prime less than $n$. Call the first $k$ primes $p_1, \dots, p_k$. Now necessarily $k$ is the largest integer such that $n\geq\sum_{i=1}^k p_i.$ Also, for every prime at most $n$, there exist $k-1$ other primes such that their sum is at most $n$. This implies that $n\geq p+\sum_{i=1}^{k-1} p_i$. Moreover, since $n < \sum_{i=1}^{k+1} p_i$, we have that $p_{k+1}+p_k>p$. \par 
Let $P(n)$ be the statement: If $k$ is the largest integer such that $$n \geq  \sum_{i=1}^k p_i,$$ then $$n < \sum_{i=1}^{k-1} p_i + p.$$ Put $f(n)=\sum_{i=1}^{k-1} p_i + p.$ We will prove by induction that $P(n)$ holds for all $n\geq 5,$ which is a contradiction to $p$ lying in a maximal simplex of size $k,$ and hence implies that the prime simplicial complex is not pure. 
For the base case, we verify directly in GAP \cite{GAP4} that $P(n)$ holds for $10\leq n\leq 129.$ 
Now suppose $P(n)$ is true, consider $P(n+1).$ If $n+1$ is prime, then $P(n+1)$ is obviously true, so we can assume that $n+1$ is not prime. Now we have three cases to consider. \par 
\underline{Case 1: $k$ and $p$ are the same for $n$ and $n+1.$}  Here we have that $f(n)=f(n+1).$ We want to show that $n+1 < f(n).$  %By $P(n),$ we have that $n < f(n),$ so it is sufficient to show that $f(n)\neq n+1.$ 
Since $n\geq 129,$ we have that $p_{k-1}\geq 25,$ so by \Cref{primes} we have that $\frac{5}{6}(n+1)<p<n+1$ and also that $p< p_k+p_{k+1}< \frac{66}{25}p_{k-1}.$ Putting these together gives that $\frac{125}{396}(n+1) < p_{k-1},$ so $$f(n)= \sum_{i=1}^{k-1} p_i + p > \frac{125}{396}(n+1) + \frac{5}{6}(n+1) > n+1, $$ as required. \par 
\underline{Case 2: $n+1 \geq  \sum_{i=1}^{k+1} p_i>n.$} This implies that $\sum_{i=1}^{k+1} p_i=n+1.$ \Cref{sumprimes} implies that $p_{k+1}<p,$ so $n+1 < \sum_{i=1}^{k} p_i +p$ and $P(n+1)$ holds.\par 
\underline{Case 3: The largest prime $p'$ such that $p'< n+1$ is not $p.$} In this case we necessarily have $p'=n.$ Since $n\geq 25,$ we have that $f(n+1)=\sum_{i=1}^{k} p_i +p'>n+1,$ so $P(n+1)$ holds.
\end{proof}

\begin{theorem}\label{anpurity}
    The prime simplicial complex of $A_n$ is pure if and only if $n\leq 5,$ $n=6$ or $n=10.$ 
\end{theorem}
\begin{proof}
    The proof is similar to that of \Cref{snpurity}, with the difference that here the order 2 elements have support of size $4,$ so $P(n)$ is the statement: If $k$ is the largest integer such that $n \geq  4+\sum_{i=2}^k p_i,$ then $n < 4+\sum_{i=2}^{k-1} p_i + p.$
\end{proof}
% \begin{proposition}
% Let $n$ be an integer. For each prime $p$ at most, we can write $n$ and $n-1$ as a sum of distinct primes, none of which are equal to $p$.
% \end{proposition}
% \begin{proof}
% We proceed by induction on $n$.
% \end{proof}
We now proceed with some results on groups of Lie type.
\begin{proposition} \label{torusprimes}
 Let $G = \mathrm{PSL}_n(q)$ be a finite simple group of Lie type in characteristic $p$. Let $r_1, \dots , r_k$ be odd primes dividing $|G|$ not equal to $p$. For each $r_i$, write $e_i = e(r_i,q)$ and suppose that each $e_i \geq 2$. Then $r_1, \dots, r_k$ do not lie in a simplex together if and only if there is no partition $n = n_1+\dots +n_m$ such that each $r_i$ divides $n_j$ for some $j$.
\end{proposition}
\begin{proof}
We observe (see \cite[Lemma A.1]{BG}) that for any odd prime $s\neq p$, we have that $s \mid q^d-1$ if and only if $e(s,q)$ divides $d$.

First suppose that there is a partition $n = n_1+\dots +n_m$ such that each $r_i$ divides $n_j$ for some $j$. Now $G$ has a maximal torus $T$ of order 
\[
\frac{1}{(n,q-1)(q-1)}(q^{n_1}-1)(q^{n_2}-1) \dots (q^{n_m}-1).
\]
In particular, $T$ is an abelian subgroup of $G$, and each prime $r_i$ divides $|T|$. Hence, $T$ contains an element of order $r_1r_2 \dots r_k$, hence these primes lie in a simplex together.

Now suppose that there is no partition $n=n_1+\dots+ n_m$ such that each $r_i$ divides $n_j$ for some $j.$   We will proceed by contradiction. Assume there is $g\in G$ such that $\vert g \vert =\prod_{i=1}^k r_i.$ Then, since $(\vert g \vert , p)=1,$ $g$ is semisimple and hence lies in a maximal torus of $G$ of order \[
\frac{1}{(n,q-1)(q-1)}(q^{t_1}-1)(q^{t_2}-1) \dots (q^{t_m}-1).
\]
for some partition $n=t_1+t_2+ \dots +t_m$.
As all $r_i$s are prime, we have that for all $r_i$, there exists $t_j$ such that $r_i\mid t_j,$ which is a contradiction. 

%Hence for every partition, there is $l$ such that  $r_l$ does not divide $n_j$ for all $j. $
\end{proof}

\begin{proposition}
Let $n\geq3,$ $p$ a prime and $q=p^e.$ If $\{r_1, r_2, \dots, r_k\}$ is a prime simplex of $\mathrm{PSL}_{n-2}(q)$ not containing $p,$ then $\{p, r_1, r_2, \dots, r_k\}$ is a prime simplex of $\mathrm{PSL}_{n}(q)$.
\end{proposition}

\begin{proof}
   We have that $\rm{PSL}_2(q)\times\rm{PSL}_{n-2}(q)\leq\rm{PSL}_n(q),$ so the result follows by \Cref{directprod}.  
\end{proof}

\begin{proposition}
If $G = \mathrm{PSL}_n(q)$ with $q=p^e$ and $r$ divides $|G|$ with $e(r,q)>n-2$, then $r$ and $p$ do not lie in a prime simplex together. That is, $G$ does not contain an element of order $pr$.
\end{proposition}
\begin{proof}
This is \cite[Proposition 3.1(1)]{MR2213302}.
\end{proof}

\begin{proposition}
\label{psl2_pure}
    The prime simplicial complex of $G=\rm{PSL}_2(q)$ for $q\geq 4$ in characteristic $p$ is pure if and only if $q=4,$ $5,$ $7,$ $8,$ $9$ or $17.$
\end{proposition}

\begin{proof}
By \cite[Proposition 3.1]{MR2213302}, $p = \mathrm{char}(\mathbb{F}_q)$ is an isolated vertex in the prime graph of $G$, hence $\Pi(G)$ is pure if and only if it has all maximal simplices of size 1. As mentioned in \Cref{ex_pure}(\ref{ex_eppo}), these are known as EPPO groups, and the result now follows from \cite[Theorem 1.7]{CAMERON2022186}.
\end{proof}

% Note that the above result also follows from \cite[Theorem 1.7]{CAMERON2022186}, as $\mathrm{PSL}_2(q)$ has a simplex of size $1$ so the prime simplicial complex is pure if and only if $\mathrm{PSL}_2(q)$ is an EPPO group. 

\begin{proposition}
\label{2b2_pure}
     The prime simplicial complex of $G={}^2\mathrm{B}_2(2^{2m+1})$ for $m\geq 1$ is pure if and only if $m\in \{1,2\}.$
\end{proposition}

\begin{proof}
    % We can check in GAP that the prime simplicial complex of ${}^2\mathrm{B}_2(2^{2m+1})$ is pure when $m\in \{1,2\}.$ Now assume $n\geq 3.$ 
    By \cite{suzuki}, centralizers of involutions in $G={}^2\mathrm{B}_2(2^{2m+1})$ are $2$-groups, so $2$ lies in a maximal simplex of size $1.$ Hence the prime simplicial complex of $G$ is pure if and only if it is an EPPO group. By \cite[Theorem 1.7]{CAMERON2022186}, this only occurs when $m\in \{1,2\}$.
    %It remains to show that there are two distinct odd primes that lie in a simplex together. By \cite[Proposition 2.6 (i)]{MR2213302}, any two distinct prime divisors of the odd numbers $Q_{\pm}=q\pm \sqrt{2q}+1$ lie in a simplex together, so the prime simplicial complex is pure if and only if they are both prime powers. By \cite[Lemma 2.14]{lee2024recognisabilitysporadicgroupsisomorphism}, $Q_{\pm}=q\pm \sqrt{2q}+1$ are both prime powers if and only if $m\in \{1,2\}.$ Hence, for $m\geq 3,$ at least one of them has at least 2 distinct prime divisors which lie in a simplex together, so the complex is not pure. 
\end{proof}
\begin{proposition}
\label{2g2pure}
     The prime simplicial complex of $G \cong {}^2\mathrm{G}_2(q)$ with $q = 3^{2m+1}$ and $m \geq 1$ is never pure. 
\end{proposition}
\begin{proof}
By \cite[Proposition 3.3]{MR2213302}, $\{3,2\}$ forms a maximal simplex, so $G$ is pure if and only if it is pure with maximal simplex size $2.$
$G$ has maximal tori of sizes $q\pm 1$ and $q \pm \sqrt{3q}+1$. Now $q \equiv 3 \mod 4$, so if $q-1$ is divisible by only two distinct primes, it must be of the form $2r^a$ for some prime $r$ and $a\geq 1$. By Lemma \ref{dio_eqn}, the only possibility is $q=3^5$, but $q+ \sqrt{3q}+1=257$ is prime, and forms a maximal simplex of size one. Therefore, $G$ is not pure.
\end{proof}

\begin{proposition}
Let $G=\mathrm{PSL}_3(q)$. If $G$ is pure then $q=p^e$ for some prime $p$ and $e\geq 1$ a power of two.
%then one of the following holds.
% \begin{enumerate}
%     \item $q=p^e$ with $e\geq 1$ a power of two.
%     \item $q>3$ is a Mersenne prime, and $|\pi(\frac{q^2+q+1}{3})| = |\pi(\frac{q-1}{3})|+1$.
% \end{enumerate}

\end{proposition}
\begin{proof}
Let $S$ be a maximal simplex containing $p$, and write $S_1 = \{r\in \mathbb{N} :  r \mid (q-1)/(3,q-1)\}$.  By \cite[Lemma 3.1.15]{BG}, we have $S = \{p\}\cup S_1$. Now $G$ has a maximal torus $T$ of size $(q^2-1)/(3,q-1)$. Since $T$ is abelian, it follows that the largest simplex corresponding to an element in $T$ has size $|S_1|+ |S_2|$, where $S_2$ is the set of primitive prime divisors of $q^2-1$. We will show that if neither of the conditions of the proposition are satisfied, then $|S_2|\geq 2$. 
We have that $q=p^e.$ Let $2^i$ be the highest power of two dividing $e,$ so $e=2^i l,$ where $l$ is odd.
If $e$ is not a power of two, then $p^{2^{i+1}}-1$ and $q^2-1 = p^{2e}-1$ are distinct and divide $q^2-1$ but not $q-1$. Each has a primitive prime divisor unless $q=8$, and we check that $G$ is impure directly.
% \textbf{why? I don't see this}, or $q$ is a Mersenne prime. In the former case, $G$ is not pure by direct calculation. 
% If $q$ is a Mersenne prime, then either $q=3$, which we check is impure by direct calculation, or  $\pi(q^2-1) = \pi(q-1)$.  In this case, the only primes not in $S$ are those lying in a torus of order $(q^2+q+1)/3$. The result follows.
\end{proof}
\begin{remark}
There are examples of pure groups of the form $\mathrm{PSL}_3(q)$ with $q=p^{2^i}$, e.g. $q \in \{4,16,25,49,121\}$. However, $\mathrm{PSL}_3(169)$ is impure, as $|\pi(\frac{q(q-1)}{3})|=3$, while $|\pi(\frac{q^2-1}{3})| = 4$. It is unknown whether there are finitely or infinitely many pure groups of this form.
% \begin{enumerate}
%     % \item A test of the first 10 Mersenne primes larger than three reveals $|\pi(\frac{q^2+q+1}{3})| < |\pi(\frac{q-1}{3})|$ in each case.
%     \item There are examples of pure groups of the form $\mathrm{PSL}_3(q)$ with $q=p^{2^i}$, e.g. $q \in \{4,16,25,49,121\}$. However, $\mathrm{PSL}_3(169)$ is impure, as $|\pi(\frac{q^2+q+1}{3})|=2$, while $|\pi((q^2-1)/(3,q-1))| = 4$. It is unknown whether there are finitely or infinitely many pure groups of this form.
% %n=4, q=107,163,191
% \end{enumerate}
\end{remark}

The following result bounds the rank of a group of Lie type that might have pure prime simplicial complex with maximal simplex size $k.$ Note that this fits into the theme of results in \cite{MR4371273}, bounding the number or primes dividing the order of a group whose elements are divisible by at most $k$ primes. 

\begin{proposition}
   Assume $G$ is a finite simple group of Lie type and has a pure simplicial complex of maximal complex size $k.$ Then one of the following holds. 
   \begin{enumerate}
       \item For $q\geq 4,$  $\mathrm{rank}(G)\leq\frac{(k+1)(k+2)}{2}-1.$
       \item For  $q\in \{2,3\},$ $\mathrm{rank}(G)\leq\frac{(k+4)(k+5)}{2}-1.$ 
   \end{enumerate}

\end{proposition}

\begin{proof}
First assume that $q\geq 4.$ 
We will prove this by exhibiting an abelian subgroup in each finite group of Lie type of rank greater than $\frac{(k+1)(k+2)}{2}$ whose order is divisible by at least $k+1$ distinct primes, as then those groups contain a maximal simplex of size $k+1.$ 
The sizes of the maximal tori in finite simple groups of Lie type are listed in \cite[Lemmas 1.2,1.3]{MR2213302}. 

First consider that case when $G$ is of type $\rm{P\Omega}_{2n+1}(q)$, $\rm{PSp}_{2n}(q)$ or $\rm{P\Omega}^\pm_{2n}(q).$ Then if $G$ contains a torus whose order is $$\frac{1}{4} (q^3-1)(q^4-1)\dots (q^{k+1}-1),$$ for $q$ odd or $$ (q^3-1)(q^4-1)\dots (q^{k+1}-1),$$ for $q$ even, then since $(q^4-1)$ by \Cref{q4lemma} has 3 distinct prime divisors, which by \Cref{zsigmondy} do not coincide with the primitive prime divisors of the other $(q^i-1)$ factors exhibited in the product, this torus is divisible by at least $k+1$ distinct primes. Note that for the division by $4$ it is necessary that $q$ is odd, hence for $k\geq 1$ the order of this torus will still be divisible by $2.$ The smallest groups containing such a torus have rank at least $\frac{k^2+3k-4}{2},$ so the result follows.

Now consider the case when $G\cong \mathrm{PSL}_n(q).$ If $G$ contains a torus whose order is $$\frac{1}{(n,q-1)(q-1)} (q-1)(q^2-1)(q^3-1)(q^4-1)\dots (q^{k+1}-1).$$ As shown before, $(q^3-1)(q^4-1)\dots (q^{k+1}-1)$ has at least $k+1$ distinct prime divisors, so $G$ is not pure of maximal simplex size $k.$ Hence $$\mathrm{rank}(G)\geq\frac{(k+1)(k+2)}{2}-1,$$ as required. 

Now consider the case when $G=\mathrm{PSU}_n(q).$ If $G$ contains a torus whose order is $$\frac{1}{(n,q+1)(q+1)} (q+1)(q^2-1)(q^3+1)(q^4-1)\dots (q^{k+1}-(-1)^{k+1}),$$ then this torus has at least $k+1$ distinct prime divisors, as the primitive prime divisor of $(q^{2i}-1),$ which exists by \Cref{zsigmondy} divides $(q^i+1)$ and not any $(q^j+1)$ for $j\leq i-1,$ and each $(q^{2l}-1)$ also provides a new prime divisor, which is its primitive prime divisor for $l$ even and the primitive prime divisor of $(q^l-1)$ for $l$ odd. Hence if $G$ has a torus with this order we get that $$\mathrm{rank}(G)\geq\frac{(k+1)(k+2)}{2}-1,$$ as required. 

Finally consider the exceptional groups. For $k=3,$ the bound trivially holds.
For $k=2,$ we see that for ${}^\epsilon \mathrm{E}_6(q)$ and $\mathrm{F}_4(q)$ there is a torus whose order is $q^4-1,$ so by \Cref{q4lemma} these have a simplex of size at least $3,$ for $\mathrm{E}_7(q)$ there is a torus with order $(q^4-1)(q^3-1)$ and for $\mathrm{E}_8(q)$ there is a torus of order $(q^4-1)(q^3-1)(q-1)$ by \cite[Lemma 1.3]{MR2213302},  so these also give rise to a simplex of size at least $3.$ The result follows. 

Now assume that $q\in \{2,3\}.$ The bound for exceptional groups holds trivially. For the classical groups we use similar methods to the case where $q\geq 4,$ but here we cannot assume that $q^4-1$ has $3$ distinct prime divisors, so we need to consider larger tori.  
If $G$ is of type $\rm{P\Omega}_{2n+1}(q)$, $\rm{PSp}_{2n}(q)$ or $\rm{P\Omega}^\pm_{2n}(q),$ then consider groups $G$ that contain a torus whose order is $$\frac{1}{4} (q^3-1)(q^4-1)\dots (q^{k+4}-1),$$ for $q$ odd or $$ (q^3-1)(q^4-1)\dots (q^{k+4}-1),$$ for $q$ even. For $q=3$ all $q^i-1$ in the order of the torus provide a new prime divisor, and for $q=2$ the only $q^i-1$ in the order of the torus that does not provide a new prime divisor is $2^6-1,$ but $q^4-1=5.3,$ so this torus will in both cases give rise to a simplex of size at least $k+1,$ so the result holds. Similarly for $G\cong\rm{PSL}_n(q),$ or $\mathrm{PSU}_n(q),$ having a torus of size  $$\frac{1}{(n,q-(\epsilon1))(q-(\epsilon1))} (q-(\epsilon1))(q^2-(\epsilon1)^2)(q^3-(\epsilon1)^3)(q^4-(\epsilon1)^4)\dots (q^{k+4}-(\epsilon1)^{k+4}),$$ will give rise to a simplex of size $k+1,$ so the result holds. 

\end{proof}

A natural question to ask is to classify groups whose prime simplicial complex is pure with small maximal simplex size. As noted in \Cref{ex_pure}(\ref{ex_eppo}), the groups with pure prime simplicial complex with maximal simplex size 1 were classified in \cite[Theorem 1.7]{CAMERON2022186}.
% The groups with pure prime simplicial complex with maximal simplex size 1 consist of prime power order elements. These groups are called EPPO groups and have been classified in \cite{MR4506711}.

We now provide a partial classification of finite simple groups of Lie type whose prime simplicial complex is pure with maximal simplex size $2.$ 

\begin{proposition}
\label{pure2}
     Assume $G$ is a nonabelian finite simple group and has a pure simplicial complex of maximal complex size $2.$ Then one of the following holds.
    \begin{enumerate}
        \item \label{l3q_pure_odd} $G \cong \mathrm{PSL}_3(q)$ with $q$ an odd prime of the form $q=3.2^k+1$ with $k\geq 1$ and $k\not\equiv 1 \mod 3$.
        \item $G \cong \mathrm{PSL}_3(q)$ with $q \in \{4,16\}$ and $G$ is pure with maximal simplex size 2. %True examples
         \item \label{unitary_1}$G \cong \mathrm{PSU}_3(q)$ with $q$ an odd prime of the form $q=3.2^k-1$ with $k\geq 1.$
        \item \label{unitary_2} $G \cong \mathrm{PSU}_3(2^e)$ with $e$ a prime, $2^e-1$ a Mersenne prime and $3\mid q+1$.
        \item $G \cong \mathrm{\rm{PSp}}_4(8), \mathrm{PSL}_3(9)$ or $A_{10},$ and $G$ is pure with maximal simplex size $2.$
    \end{enumerate}
    
    % ${}^\epsilon A_2(q)$ or  ${}^2 G_2(q).$
\end{proposition}

\begin{proof}
First consider the case when $G$ is sporadic or alternating. Then, by \Cref{sporpurity} and \Cref{anpurity} we have $G\cong A_{10}.$
From now on assume that $G$ is a finite simple group of Lie type. 
We begin by supposing $q\geq 4$; the cases $q=2,3$ will be dealt with separately at the end.
    The classical groups with rank at least $4,$ with the exception of $\rm{PSL}_5(q)$ and $\mathrm{PSU}_5(q),$ do not arise, as they contain a torus whose order is divisible by $\frac{q^4-1}{(4,q-1)},$ hence by \Cref{q4lemma} they have a simplex of size at least $3.$  
   Now suppose $G$ is isomorphic to one of the following: $\cong \mathrm{G}_2(q),$  $\rm{P\Omega}_5(q),$ $\rm{PSp}_4(q),$  $\rm{P\Omega}_7(q),$ $\rm{PSp}_6(q),$ $\rm{PSL}_4(q),$ $\mathrm{PSU}_4(q),$ $\rm{PSL}_5(q)$  or $\mathrm{PSU}_5(q).$ All of these groups have a torus of size divisible by $\frac{q^2-1}{4},$ so by \Cref{q2_divisors} they contain a simplex of size $3$ unless $q=4,$ $5,$ $7,$ $8,$ $9$ or $17.$ In these cases we check by direct calculation that the result holds. 
    % It remains to consider groups of one of the following types: 
    
    By Propositions \ref{psl2_pure} and \ref{2b2_pure}, no groups isomorphic to $\rm{PSL}_2(q)$ or ${}^2\mathrm{B}_2(q)$ arise, since each has a maximal simplex of size one.

  %  \bf{I can prove that $SL_3(q)$ has a simplex of size $3,$ but I cannot for $\rm{PSL}_3(q).$ I think this is fine, but how to word this???}

    Now suppose that $G \cong {}^3 D_4(q)$ or $\mathrm{F}_4(q)$. Then $G$ has a torus $T$ of size $|T| = (q^3- 1)(q+ 1)$. If $q\neq 8$, $|T|$ is divisible by the (distinct) primitive prime divisors of $q-1$, $q^2-1$ and $q^3-1$, hence has a simplex of size at least 3. If $q=8$, a direct calculation shows $|T|$ is divisible by three primes.

    Similarly, suppose that $G \cong {}^2\mathrm{F}_4(q)$ with $q=2^{2m+1}$ and $m\geq 2.$ Now $G$ has a torus of order $(q^2-1)=(q-1)(q+1),$ which has at least $2$ distinct prime divisors by Lemma \ref{zsigmondy}. If $(q^2-1)$ had exactly $2$ distinct prime divisors, then we can conclude that $q-1$ and $q+1$ are prime powers, hence they are primes unless $q=8.$ If $q=8,$ a direct calculation shows that $G$ is not pure of a maximal simplex size $2$. If $q-1$ and $q+1$ are both primes, then $q-1$ is a Mersenne prime and $q+1$ is a Fermat prime, so $2m+1=2,$ which is a contradiction, so $G$ is not pure of maximal simplex size $2.$ 

Now suppose $G \cong\rm{PSL}_3(q)$ with $q=p^e$ with $e\geq 1$. Then $G$ has tori of orders $(q-1)^2/\alpha$, $(q^2-1)/\alpha$ and $(q^2+q+1)/\alpha$, where $\alpha = (3,q-1)$. We first consider the case where $q$ is even. In this case, $(q^2-1)=(q-1)(q+1)$ with $q-1$, $q+1$ coprime. 
If $(q^2-1)/\alpha$ has exactly two distinct prime divisors, then either $q=8$ (and not pure by direct computation), or $q+1$ must be a Fermat prime. In the latter case, $e=2^i$ by Lemma \ref{2-3prime} and $(q^2-1)$ is divisible by the primitive prime divisors of $2^{2^j}-1$ for each $1\leq j \leq i+1$. In particular, if $i>2$, then $(q^2-1)$ is divisible by at least four distinct primes and the result follows. We check directly that $\rm{PSL}_3(4)$ and $\rm{PSL}_3(16)$ are pure with maximal simplex size 2.

Now assume $q=p^e$ is odd. If $q$ is not prime, then $(q^2-1)$ is divisible by the primitive prime divisors of $p-1$, $p^2-1$, $p^e-1$, and $p^{2e}-1$. At least three of these are distinct unless $q=9$, in which case the group is pure with maximal simplices of size 2. So suppose that $q$ is prime.

If $\alpha=1$, then by Proposition \ref{q2_divisors}, $q\in \{5,17\}$. A direct calculation shows that $\Pi(\rm{PSL}_3(5))$ and $\Pi(\rm{PSL}_3(17))$ have maximal simplices $\{37\}$ and $\{307\}$ respectively and hence are not pure. %Also 7, but \alpha \neq 1

So now suppose $\alpha=3$, so $3\mid q-1$. If $(q^2-1)/3$ has only two distinct prime divisors, then we must have $q-1 = 3.2^k$ and $q+1 = 2^l.r^m$ for some prime $r$ and $k,l,m \in \mathbb{Z}_{>0}$. If $k \equiv 1\mod 3$, then $q$ is divisible by 7, hence $q=7$ as it is prime. But $\Pi(\rm{PSL}_3(7))$ has a maximal simplex $\{19\}$, hence is not pure.

Now suppose $G \cong \mathrm{PSU}_3(q)$ with $q=p^e$ with $e\geq 1$. Then $G$ has tori of orders $(q+1)^2/\beta$, $(q^2-1)/\beta$ and $(q^2-q+1)/\beta$, where $\beta = (3,q+1)$.  

First assume that $(q^2-1)/\beta$ and $q^2-1$ have the same set of prime divisors, so by \Cref{q2_divisors}, $(q^2-1)/\beta$ has at least $3$ distinct prime divisors unless $q\in \{4,5,7,8,9,17\}.$ In the latter case we check by direct computation that $G$ is not pure.

%that if $q\in \{4,5,7,8,9,17\},$ then $G$ is not pure.
%so either $q+1$ is divisible by $3^k$ where $k\geq 2$ or not divisible by $3.$ If $q$ is even and $q^2-1$ has exactly two distinct prime divisors, then either $q-1$ is a Mersenne prime and $q+1$ is a Fermat prime, or $q=8.$ The former can only happen if $q=4.$ We check by direct computation that $\mathrm{PSU}_3(4)$ and $\mathrm{PSU}_3(8)$ are not pure. If $q$ is odd, then one of $q-1$ or $q+1$ is a power of two, so $q$ is either a Fermat or a Mersenne prime or $q=9$. However, in the former case $\frac{q+1}{2}$ is also a Fermat or Mersenne prime, respectively, so $q=4,$ in which case we know that $G$ is not pure.  We check by direct computation that $\mathrm{PSU}_3(9)$ is not pure either. 

Now assume that $q$ is odd and that $(q^2-1)/\beta$ and $q^2-1$ do not have the same set of prime divisors, so $\beta=3$ and $q+1=3r$ such that $(3,r)=1.$ We may assume that $q+1=3.2^k,$ as otherwise $(q^2-1)/\beta$ would have at least 3 distinct prime divisors. If $e\geq 2,$ then $(q^2-1)/\beta=\frac{(p^e-1)(p^e+1)}{3}.$ Here $(p^e-1)(p^e+1)$ has at least $3$ distinct prime divisors, the primitive prime divisors of $p-1,$ $p^e-1$ and $p^{2e}-1.$ Now $p-1$ is either divisible by two distinct primes, in which case $(q^2-1)/\beta$ has at least $3$ distinct prime divisors, so $G$ is not pure of maximal simplex size $2,$ or $p-1$ is a power of two. Hence $p$ is a Fermat prime, and $p-1=2^{2^x}$ for some $x\geq 0.$ Now $p+1$ is divisible by $2$ and $3,$ and if it had any other prime divisors we would get a simplex of size at least $3,$ so we can assume that $p+1=3.2^l.$ However, since $\frac{p-1}{2}$ and $\frac{p+1}{2}$ are consecutive numbers, it follows that $l=1,$ so $p=5.$ Now if $p=5$ and $e$ is even, then $3$ divides $5^e-1,$ so we do not need to consider that case, as then $(3,q+1)=1$. If $e=2m+1$ for some $m,$ then $5^{2m+1}+1$ has a factor of $5+1=6,$ so $5^{4m+2}-1$ has a primitive prime divisor $r$ not equal to $2$ or $3.$ Also $5^{2m+1}-1$ has a primitive prime divisor $s$ not equal to $2,$ $3$ or $r,$ so $(q^2-1)/\beta$ has at least $3$ distinct prime divisors and gives rise to a simplex of size $3$. Hence $q$ is prime.

Now assume that $q=2^e$ and that $(q^2-1)/\beta$ and $q^2-1$ do not have the same set of prime divisors, so $\beta=3$ and $q+1=3r^k$ such that $(3,r)=1.$ If $e$ is even, then $2^e-1$ is divisible by $3,$ so $(3,q+1)=1,$ so we do not need to consider this case. If $e$ is odd, then $2^e-1$ is a prime power, hence a Mersenne prime, so $e$ is prime.  

Suppose that $G \cong {}^2\mathrm{G}_2(q)$ with $q = 3^{2m+1}$ and $m \geq 1$. Then by \Cref{2g2pure}, $G$ is not pure.  
%Then $G$ has maximal tori of sizes $q\pm 1$ and $q \pm \sqrt{3q}+1$. Now $q \equiv 3 \mod 4$, so if $q-1$ is divisible by only two distinct primes, it must be of the form $2r^a$ for some prime $r$ and $a\geq 1$. By Lemma \ref{dio_eqn}, the only possibility is $q=3^5$, but $q+ \sqrt{3q}+1=257$ is prime, and forms a maximal simplex of size one. Therefore, $G \not\cong {}^2G_2(q)$.

% If $2m+1$ is composite with proper divisor $k$, then $q-1$ is divisible by 2, as well as the (distinct) primitive prime divisors of $3^k-1$ and $3^{(2m+1)}-1$, so has a maximal simplex of size at least 3. Hence $2m+1$ is prime.

Finally, suppose that $q\in \{2,3\}$. We check directly that the groups in the following list each contain an element of order a product of three primes:
\[
\{\mathrm{PSL}_7(2),\mathrm{PSL}_5(3), \mathrm{PSU}_7(2), \mathrm{PSU}_5(3),\mathrm{\rm{PSp}}_8(2), \mathrm{\rm{PSp}}_6(3),\mathrm{P}\Omega_8^-(3),\mathrm{P}\Omega^-_8(2)\}.
\]
Each classical group over $\mathbb{F}_2$ or $\mathbb{F}_3$ of larger rank contains a subgroup isomorphic to one of these groups, or a central extension of one. Hence, each also has an element of order a product of three primes, and so cannot arise. We compute directly that no classical groups of smaller ranks with $q=2,3$ are pure with a maximal simplex of size 2. If $G$ is instead isomorphic to an exceptional group, we either check directly that $G$ has a maximal simplex of size not equal to 2, or use \cite{craven2023maximal} to show that $G$ has a maximal subgroup with an element of order a product of three distinct primes.
\end{proof}

\begin{remark}

   There are genuine examples and non-examples in Proposition \ref{pure2}(\ref{l3q_pure_odd}). For example, if $k=2,5$ so that $q=13,97$, then $(q^2+q+1)/3$ is prime, so $\Gamma(\rm{PSL}_3(q))$ has a maximal simplex of size 1. If instead $k=6$, we note that $\Gamma(\rm{PSL}_3(193))$ is pure with maximal simplex size of 2. It is unknown where there are any pure examples arising in parts (\ref{unitary_1}) or (\ref{unitary_2}).

\end{remark}

\bibliographystyle{abbrv}
\bibliography{refs.bib}
\end{document}